\documentclass[12pt, oneside]{article}

\usepackage{geometry}                		% See geometry.pdf to learn the layout options. There are lots.
\usepackage{authblk} 
\usepackage{comment}
\usepackage{amssymb, amsthm, mathtools,bbm}
\newcommand{\rr}{{\mathbb{R}}}

\newcommand{\zz}{{\mathbb{Z}}}
\newcommand{\nn}{{\mathbb{N}}}

\newcommand{\supp}{{\operatorname{supp}\,}}

\newcommand{\beq}[1]{\begin{equation} \label{#1}}
\newcommand{\eeq}{\end{equation}}

\providecommand{\keywords}[1]
{
	\small	
	\textbf{\textit{Keywords---}} #1
}

\newtheorem{theorem}{Theorem}
\newtheorem{corollary}[theorem]{Corollary}
\newtheorem{lemma}[theorem]{Lemma}
\newtheorem{definition}[theorem]{Definition}
\newtheorem{proposition}[theorem]{Proposition}

\numberwithin{equation}{section}
\numberwithin{theorem}{section}

\begin{document}

%\begin{frontmatter}


%\title{No Homeomorphism Maps Non-Normal  to Normal Numbers \tnoteref{tit}}

\title{Super-Dense Sets and Their Role in the Theory of Normal Numbers}


%\author{Chokri Manai} %% Author name

%\affiliation{organization={Department of Mathematics, Technical University of Munich},%Department and Organization
           % addressline={Boltzmannstr.~3}, 
            %city={Garching b. München},
            %postcode={85748}, 
 %           state={},
            %country={Germany}}
\author[1]{Chokri Manai}
%\affil[1]{\small School of Physics and Astronomy, University of Nottingham, Nottingham, NG7~2RD, UK}
%\affil[2]{\small Centre for the Mathematics and Theoretical Physics of Quantum Non-Equilibrium Systems,
	%%University of Nottingham, Nottingham, NG7 2RD, UK}
\affil[1]{\small  Courant Institute of Mathematical Sciences, New York University, USA}
%\affil[2]{\small Munich Center for Quantum Science and Technology, Munich, Germany}
\date{\today}			

\maketitle
%% Abstract
\begin{abstract}
\noindent We introduce and study a new topological notion of the size of subsets of the real line, called \emph{super-density}. A set $A\subset\mathbb{R}$ is super-dense if for every non-empty open interval $I$ and every nowhere constant continuous function $\varphi\colon I\to\mathbb{R}$, we have $\varphi(I\cap A)\cap A\neq\emptyset$.

\noindent We first establish basic properties of super-dense sets. Our main topological result characterizes them within the framework of Baire category: a set with the Baire property is super-dense if and only if it is co-meager.

\noindent We then investigate the implications for the theory of normal numbers. Our topological considerations imply that no nowhere constant continuous function can map all non-normal numbers to normal numbers. Conversely, we explicitly construct a computable nowhere constant continuous function that maps all normal numbers to non-normal numbers.

\noindent Finally, we provide a constructive algorithm that, given any countable family of nowhere constant continuous functions, produces a real number $x$ such that $x$ and all its images under these functions are non-normal. As a corollary, we obtain the existence of a non-normal number $x$ such that $e^{\alpha x}$ is non-normal for every non-zero algebraic $\alpha$.
\end{abstract}



%% Keywords
%\begin{keyword}


%\end{keyword}

%\end{frontmatter}




\section{Introduction and Main Result}~\label{sec:intro}
Ever since \'{E}mile Borel introduced the concept of normality and showed that almost every number is normal \cite{Bor09}, generations of mathematicians have worked hard to unravel the secrets of normal numbers. Although many profound statements about the set of normal numbers have been shown \cite{Cas59, Bug12, Lev79, Rau76, Sir17, Schm62, Man25} and in recent years there has been increasing interest in the computability of concrete normal numbers \cite{ABSS17,BF02,BHS13}, the most important conjectures concerning normal numbers remain unsolved. In particular, it currently seems infeasible to show that important mathematical constants like $\sqrt{2}, e,$ or $ \pi$ are normal. 
By contrast, transcendental number theory has witnessed significant breakthroughs. Although many conjectures such as Schanuel's conjecture regarding transcendental numbers remain open, Hermite, Lindemann and Weierstrass proved in their seminal works the transcendence of $e$ and $\pi$ over a century ago \cite{Her73,Hil93,Lin82,Wei85}. In fact, the celebrated Lindemann-Weierstrass theorem states that the exponential $e^{\alpha}$ is transcendental for every algebraic number $\alpha \neq 0$. Our main results imply that no similar statement holds in the context of (non-)normal numbers.
Indeed, we construct a non-normal number $x$ such that all exponentials $e^{\alpha x}$ are not normal for every algebraic number $\alpha$. 

Our main insight for the construction of non-normal numbers is to recall that despite being a set of Lebesgue measure zero, the non-normal numbers form a co-meager set and are thus topologically big \cite{AL23, Ol04a, Ol04b}. In this work, we introduce the closely related topological concept of super-density which is of independent interest and will allow us to prove non-normal numbers with desired properties. First, let us recall the definition of a nowhere constant function.

\begin{definition}\label{def:nocon}
	Let $I \subset \rr$ be a non-empty open interval and $\varphi : I \to \rr$ some function. We say that $\varphi$ is non-constant at $t_0 \in I$ if for every $\varepsilon > 0 $
	\begin{equation}
		\varphi((t_0 - \varepsilon, t_0 + \varepsilon)) \setminus\{\varphi(t_0) \} \neq \emptyset.
	\end{equation}
	The function $\varphi$ is nowhere constant on $I$ if it is  non-constant at every $t \in I.$
\end{definition}

We remark that for a function $\varphi$ to be nowhere constant on some interval $I$, it  suffices that the set
$$ D := \left\{ x \in I \, \bigg| \, \liminf_{h \to 0} \left|\frac{\varphi(x+h) - \varphi(x)}{h}\right| > 0 \right\} $$
is dense in $I$. In particular, a Brownian motion is almost surely (with respect to the canonical Wiener measure) a nowhere constant continuous function.
Our main novel concept is content of the following

\begin{definition}\label{def:super}
	We call a set $A \subset \rr$ super-dense, if for every non-empty open interval $I$ and every nowhere constant continuous  function $\varphi : I \to \rr$ one has
	\begin{equation}
		\varphi(I \cap A) \cap A \neq \emptyset. 
	\end{equation}
\end{definition}

The definition of super-density seems not to have been formulated in the mathematical literature before. However, certain constructions are based on related ideas.  For instance, super-density is  related to the definition of a universally meager set in  \cite{Tod07}. Clearly, a super-dense set is dense in the usual sense with respect to the Euclidean metric; and it is also immediate that a super-dense set is necessarily uncountable. Hence, super-dense sets need to be topologically big and the natural question arises to understand the relation between super-density and the well-studied Baire categories. We briefly recall that a real subset $A \subset \rr $ is meager or of first category if it is given by an at most countable union of nowhere dense sets. Otherwise, $A$ is  of second category or a fat / non-meager set. If the complement $A^c := \rr \setminus A$ is meager, $A$ is co-meager. Finally,  a set $A$ has the Baire property if and only if there is an open set $U$ and a meager set $N$ such that $A = U \Delta N$, where $\Delta$ denotes the symmetric difference. Our main result below gives a satisfactory relation between super-density and the Baire categories.

\begin{theorem}\label{thm:topology}
	The following hold.
	\begin{enumerate}
		\item Suppose $A \subset \mathbb{R}$ has the Baire property. Then, $A$ is  super-dense if and only if $A$ is co-meager.
		\item There exists a set $A \subset \mathbb{R}$ such that both $A, A^c$ are dense, but not super-dense.
		\item There exists a set  $A \subset \mathbb{R}$ such that both $A, A^c$ are  super-dense. 
	\end{enumerate}
\end{theorem}

The proof of Theorem~\ref{thm:topology} is spelled out in Section~\ref{sec:prooftop}. We note that Theorem~\ref{thm:topology} implies the following chain of implications for general sets $A$:
\begin{equation}
	A \text{ co-meager } \Rightarrow A\text{ super-dense } \Rightarrow A \text{ non-meager.}
\end{equation}
This follows from the observation that co-meager and meager sets always have the Baire property. Super-density is in general not stable under intersection, that is, there exist super-dense sets $A,B$ such that $A \cap B$ is not super-dense itself. This is immediate from the third assertion. 

The set $A$ from the third assertion is an example for a super-dense set not being co-meager. In view of the first assertion this set $A$ is also an example for a super-dense set without the Baire property.  
By the Baire category theorem, at least one of the sets $A$ or $A^c$ from the second assertion is non-meager and, hence, an example for a set of second category which is not  super-dense.
\vspace{0.3cm}

To formulate the implications of Theorem~\ref{thm:topology} for normal numbers, we recall first some standard terminology. A real number $x \in \rr$ has a unique expansion with respect to an integer base $b \geq 2$ in the form $x = X_1 X_2 \ldots X_K. x_1 x_2 \cdots$ with digits $X_i, x_i \in \{0,1,\ldots, b-1\}$. This $b$-ary expansion of $x$ is unique if we agree to use a finite expansion, i.e., an expansion with only finitely many nonzero digits,  whenever possible. We denote by $\Lambda_{b,d,M}(x)$ the relative fraction of the digit $d$ among the first $M$ fractional $x_i$ in base $b$ (see also \eqref{eq:Lambda}). A real number $x$ is called \textit{simply} $\mathit{b}$-\textit{normal} if  $\lim_{M \to \infty} \Lambda_{b,d,M}(x) = \frac1b$, that is if all digits are asymptotically equidistributed in the $b$-ary expansion of $x$. If not just the digits but all finite words $w \in \bigcup_{k = 1}^{\infty}\{0,1,\ldots, b-1\}^k$ appear with asymptotic frequency $b^{-|w|}$ in the $b$-ary expansion of $x$, we call $x$ a $\mathit{b}$-\textit{normal} number.
Finally, $x$ is an \textit{absolutely normal} number or just a \textit{normal} number if $x$ is $b$-normal with respect to all integer bases $b \geq 2$.  

We denote  by $\mathcal{N}_{b,s} = \{ x \in \rr \, | \, x \text{ is simply } b\text{-normal} \}$ the set of all $b$-simply normal numbers for some fixed integer base $b \geq 2$. Analogously, we introduce the set of all $b$-normal real numbers $\mathcal{N}_b$, and the set of all (absolutely) normal numbers $\mathcal{N}$. Finally, we introduce the set 
\begin{equation}\label{eq:zero1}
	\mathcal{Z}_{b} := \{x \in \rr \, | \, \limsup_{M \to \infty } \Lambda_{b,0,M}(x) = 1 \},
\end{equation}
which contains all numbers which accumulate zeros in their $b$-adic expansion. Theorem~\ref{thm:topology} implies the following
\begin{corollary}\label{cor:abstr}
	The following hold.
	\begin{enumerate}
		\item The set of normal numbers $\mathcal{N}$ is not super-dense.
		\item The sets of non-normal numbers $\mathcal{N}^c$ and $\mathcal{Z}_{b}$ are super-dense for each integer base $b \geq 2$.
		\item $\mathcal{Z}_{b}$ has zero Hausdorff dimension for each integer base $b \geq 2$, whereas $\mathcal{N}^c$  has full Hausdorff dimension of $1$.
	\end{enumerate}
\end{corollary}

\begin{proof}
	That $\mathcal{N}^c$ has Hausdorff dimension $1$ is well known \cite{Egg49} and the proof that $\mathcal{Z}_{b}$ has zero Hausdorff dimension is presented in the appendix in form of Lemma~\ref{lem:dim}. In view of the first statement of Theorem~\ref{thm:topology}, it is enough to show that 
	$\mathcal{Z}_{b}$ is co-meager since $\mathcal{Z}_{b} \subset \mathcal{N}^c$ directly implies that $\mathcal{N}^c$ is then  co-meager and super-dense, too. Moreover, $\mathcal{N}$ is meager, hence not super-dense. The following argument is folklore and presented to be self-contained - refined and more subtle analysis can be found in \cite{AL23, Ol04a, Ol04b}. 
	
	We fix an integer base $b \geq 2$ and  for $x \in \rr$ we denote by $x_i \in \{0, \ldots, b-1\}$ the decimal places of $x$. We define
	$$ G_n := \bigcup_{k \geq n} \{x \in \rr \setminus \mathbb{Q} \, | \, x_m = 0 \text{ for all } m = k, \ldots, k^2  \} =: \bigcup_{k \geq n} G_{n,k} $$
	The set $G_n$ is clearly dense. It is also open since each $G_{n,k}$ is a finite union of intervals with rational boundary points. Hence, $G_n^{c}$ is nowhere dense and \begin{equation}\label{eq:Gb}
		G_b := \bigcap_{n \geq 1} G_n
	\end{equation} is a co-meager $G_{\delta}$ set. Finally we note that $G_b \subset \mathcal{Z}_b \subset \mathcal{N}_{b,s}^c. $
\end{proof}

We see that super-density and the measure-theoretical size of a set are independent of each other. Corollary~\ref{cor:abstr} says: \textit{There is no nowhere constant continuous function that maps all non-normal numbers to normal numbers}. One can slightly sharpen Corollary~\ref{cor:abstr}  the ``absolutely non-normal" numbers $\bigcap_{b = 2}^{\infty}  \mathcal{N}_{b,s}^{c}$ are also super-dense. This result is presented in the appendix as Corollary~\ref{cor:allb}. 

The proof of Theorem~\ref{thm:topology} is non-constructive. Thus, we do not obtain a procedure to compute a non-normal $x$ such that $\varphi(x)$ is also non-normal for a given nowhere constant continuous function $\varphi$. The explicit construction of such numbers is content of our next result. 

\begin{theorem}\label{thm:main}
	Given a non-empty open interval  $I \subset \rr$, a  nowhere constant continuous  function $\varphi \colon I \to \rr$ and integer base $ b \geq 2$, there exists an iterative procedure which computes a real number $x \in I $ such that
	\begin{equation}\label{eq:main}
		x \in \mathcal{Z}_b \, \, \text{ and } \, \, \varphi(x) \in  \mathcal{Z}_b
	\end{equation}
	In fact, we have more generally an iterative procedure, which, given a countable collection $(\varphi_k)_{k \in \nn}$ of nowhere constant continuous functions $\varphi_k \colon I \to \rr$, determines a real number $x \in I$ such that
	\begin{equation}\label{eq:mainfamily}
		\varphi_k(x) \in \mathcal{Z}_b \quad \text{ for all } k \in \nn
	\end{equation}
	
\end{theorem}

The proof of Theorem~\ref{thm:main} is spelled out in Section~\ref{sec:proof1}. Our method is constructive and the output is computable in the sense that under a few regularity assumptions on the functions $\varphi_k$ we can infer from the proof an algorithm computing inductively the digits of a real number $x$ such that all numbers $\varphi_k(x)$ are non-normal. The details are presented in Corollary~\ref{cor:algo} below. 

We conjecture that one can further show that there is in fact no non-constant continuous function $\varphi$ such that $\varphi(\mathcal{N}^{c}) \cap \mathcal{N}^{c} = \emptyset$. However, our constructive proof fails in this situation and it is unclear to us if the statement of  Theorem~\ref{thm:main} still holds for non-constant (but not necessarily nowhere constant) functions. 
We hope to address this question in future research.

Let us discuss a specific consequence of Theorem~\ref{thm:main}. We denote by $\bar{\mathbb{Q}}$ the set of real algebraic numbers. Applying Theorem~\ref{thm:main} to the countable collection $(e^{\alpha x})_{\alpha \in \bar{\mathbb{Q}} \setminus \{0\}}$  and to all open intervals $I = (\alpha, \beta) \subset \rr$, one concludes that the set  $$ A := \{ x \, | \, e^{\alpha x} \text{ not normal for all } \alpha \in \bar{\mathbb{Q}} \setminus \{0\} \}$$ is a dense subset of $\rr$. This shows that  there is no analog of the Lindemann-Weierstrass theorem for normal numbers. As far as we know, this work is the first to give a constructive method  to generate a non-normal number $x \neq 0$ for which $e^x$ is not normal.

Theorem~\ref{thm:main} complements the recent work \cite{Man25}, in which a number $x$ is constructed such that all numbers $f_k(x)$ are normal when a collection of maps $(f_k)_{k \in \nn}$ is given. In the future, we hope to address constructions involving normality and non-normality. For instance, it would be very interesting to determine a normal number $x$ such that $x^2$ is not normal - mimicking the conjectured property of $\sqrt{2}$.
	
\vspace{0.3cm}

The normal numbers $\mathcal{N}$ form the archetype of a set of full measure which is topologically small, i.e., meager and hence not super-dense. It is thus of interest to understand if in this case a map $\varphi$ contradicting super-density relies on nonconstructive tools. It turns out that we are able to explicitly construct a nowhere constant continuous function witnessing that the set of normal numbers is not super-dense.

\begin{theorem}\label{thm:normal}
There is an explicit nowhere constant continuous function $\widehat{C} : \rr \to \rr$ such that 
\begin{equation}
	\widehat{C}(\mathcal{N}) \cap \mathcal{N} = \emptyset.
\end{equation}
\end{theorem}

The function $\widehat{C}$ consists of infinitely many self-similar copies of the well-known Cantor function. The explicit definition and a proof of the desired properties will be given in Section~\ref{sec:normal}. It might be of interest to further understand if $\hat{C}$ is computable as a function, e.g. Type-2 computable \cite{Wei00}. We will not go into this question further here, as it goes beyond the main subject of this paper.

%% Labels are used to cross-reference an item using \ref command.

%% Use \subsection commands to start a subsection.
\section{Proof of Main Results}\label{sec:proof}

\subsection{Proof of Theorem~\ref{thm:topology}}\label{sec:prooftop}

In this section, we  prove all statements of Theorem~\ref{thm:topology} one by one. The proofs are all independent of each other and are therefore presented separately. The proof of the first assertion requires several preparatory results. We start by showing that super-dense sets cannot be meager.

\begin{proposition}\label{prop:top1}
	Let $A \subset \mathbb{R}$ be super-dense. Then, $A$ is also non-meager.
\end{proposition}
Since the set of normal numbers $\mathcal{N}$ is meager,  it is not super-dense.
\begin{proof}
	We proceed by contraposition. We fix some meager set $A \subset \mathbb{R}$ and we want to show that $A$ is not superdense. Our argument is nonconstructive. Indeed, let $I = [0,1]$ be the unit interval, $C(I)$ the space of continuous functions on $I$ - equipped with the uniform topology - and consider the sets 
	\begin{align}
		Y &:= \{ f \in C(I) \, | \, f \text{ is constant somewhere} \}, \\
		Z & :=  \{ f \in C(I) \, | \, f(I \cap A) \cap A \neq \emptyset \} \label{eq:setZ}.
	\end{align}
	Since $C(I)$ is a complete metric space, it is enough to show that $Y$ and $Z$ are meager as then Baire's category theorem implies that there exists a nowhere constant $f \in C(I)$ such that $f(I \cap A) \cap A = \emptyset$. Extending such an $f$ linearly on the whole real line, proves that $A$ is not super-dense. 
	
	We first show that $Y$ is meager. 
	If a function $f$ is constant at some $t$, there exists an interval $[q_1, q_2]$ with rational endpoints on which $f$ is constant. Thus,
	$$ Y = \bigcup_{q_1 < q_2 \in [0,1] \cap \mathbb{Q}} \{  f \in C(I) \, | \, f  \text{ is constant on } [q_1, q_2] \} = \bigcup_{q_1 < q_2 \in [0,1] \cap \mathbb{Q}} Y_{q_1, q_2}, $$
	where we used the abbreviation $ Y_{q_1, q_2} := \{  f \in C(I) \, | \, f  \text{ is constant on } [q_1, q_2] \}.$ It suffices to show that each $Y_{q_1, q_2}$ is nowhere dense. Clearly, $Y_{q_1, q_2}$ is a closed set in $C(I)$ and for any $f \in Y_{q_1, q_2}$ and any $\varepsilon > 0$, the function $g_\epsilon(x) = f(x) + \varepsilon x \notin Y_{q_1, q_2} $. Hence, $Y_{q_1, q_2}$ is nowhere dense.

	Suppose now that $A \subset \mathbb{R}$ is in fact closed and nowhere dense. We show that then $Z$ defined as in \eqref{eq:setZ} is nowhere dense (in $C(I)$), too. We fix some $f \in C(I)$ and $\varepsilon > 0$. Since $A$ is nowhere dense, we find for every $x \in I$ a $u_x \in A^c$ such that $|u_x - f(x) | < \varepsilon$. Due to the continuity of $f$, the sets 
	$$ U_x := \{ y \in I \, | \, |f(y) - u_x | < \varepsilon \}$$
	are (relatively) open and let $\widehat{U}_x$  further be the connected component of $ U_x$ containing $x$. Then, the collection $(\widehat{U}_x)_{x \in I}$ forms an open cover of $I$. So, by compactness,  there exists a finite cover  $\widehat{U}_{x_1}, \ldots \widehat{U}_{x_n}$ of $I$ with $x_1 <x_2 < \cdots < x_n$. Using again the assumption that $A$ is closed and nowhere dense, we can find open intervals $J_{x_i, x_{i+1}} \subset (x_i, x_{i+1})$ such that $J_{x_i, x_{i+1}} \subset A^c$ for $i = 1, \ldots, n-1$. 
	Finally, we set $I_{x_i} := \widehat{U}_{x_i} \setminus \left(\bigcup_{j = 1}^{n-1} J_{x_j, x_{j+1}}\right)$ for $i = 1, \ldots, n$. The intervals $I_{x_i}, J_{x_i, x_{i+1}}$ form a partition of $I$ and we define the function $g$ to be $g \equiv u_{x_i}$ on $I_{x_i}$ and the linear interpolation between $u_{x_i}$ and $u_{x_{i+1}}$ on $J_{x_i, x_{i+1}}$. The function $g$ is clearly continuous and by construction $\| f - g \| < \varepsilon$ and $g(A \cap I) \cap A = \emptyset$. Since $\varepsilon > 0$ is arbitrary, $Z$ is nowhere dense.
	
	We turn to the general case that $A$ is meager. Then, we find a countable family of nowhere dense sets $F_k \subset \mathbb{R}$ such that $A = \bigcup_{k = 1}^{\infty} F_k$.   For $K \in \mathbb{N}$, we define $G_K := \bigcup_{k = 1}^K \bar{F}_k$ and 
	$$ Z_K := \{ f \in C(I) \, | \, f(I \cap G_K) \cap G_K \neq \emptyset \}. $$
	Here, we used the standard notation $\bar{D}$ for the topological closure of an arbitrary set $D \subset \rr$. 
	We clearly have $Z \subset \bigcup_{K = 1}^{\infty} Z_K$ and by our previous considerations $Z_K$ is nowhere dense. Hence, $Z$ is a meager subset of $C(I)$.
\end{proof}

In fact, we can further strengthen the assertion of Proposition~\ref{prop:top1}. Recall that a set $A \subset \rr$ is said to be \textit{nowhere meager} if $A \cap I$ is non-meager for every open interval $I \subset \rr$.

\begin{corollary}\label{cor:top1}
	Let $A \subset \mathbb{R}$ be super-dense. Then, $A$ is nowhere meager.
\end{corollary}

\begin{proof}
	We proceed via contradiction. Suppose $A \cap I$ is meager for an open interval $I$. Super-density is invariant under translation and stretching. Hence, we may assume without loss of generality that $ B: = A \cap (0,1)$ is meager. We set 
	$$ \widehat{B} := B \cup (1-B) = B \cup \{ 1 -x \, | \, x \in B \} $$
	and further define the periodized set
	$$ \mathcal{B} := \{x + n \, | \, x \in \widehat{B}, n \in \zz  \}. $$
	It immediately follows that $\mathcal{B}$ is meager, too. By Proposition~\ref{prop:top1}, $\mathcal{B}$ is not super-dense and, thus, there exists a nowhere constant continuous function $\varphi : \rr \to \rr$ and an open interval $J \subset \rr$ such that 
	$$ \varphi( J \cap \mathcal{B}) \cap \mathcal{B} = \emptyset. $$
	Due to the periodicity of $\mathcal{B}$, we may assume that $J \subset (0,1)$. Consider the chainsaw-type function
	$$ g(x):= \sum_{n \in 2 \zz} | x - n| \mathbbm{1}_{[-1,1)}(x-n), $$
	which is exactly the periodic continuation of the absolute value from $[-1,1)$ to $\rr$. We observe that
	$g$ is continuous and $g \circ \varphi$ is still nowhere constant. Furthermore,
	$$ (g \circ \varphi)( J \cap \mathcal{B}) \cap \mathcal{B} = \emptyset. $$
	The reason is that we can write
	$$ g(x) = \begin{cases}
		x - 2n & \text{ if } x \in [2n, 2n +1], \\
		1 - (x - 2n-1) & \text{ if } x \in [2n+1, 2n +2] ,
	\end{cases} $$
	for $n \in \zz$ and the claim follows from the periodicity and reflection invariance of $\mathcal{B}$. We further shrink $J$ such that $\{0,1\} \notin (g \circ \varphi)(J)$. This is always possible due to continuity and the nowhere constancy of $g \circ \varphi$.  Since $g \circ \varphi$ takes only values in $(0,1)$ on $J$ we obtain
	$$ (g\circ \varphi)( J \cap A) \cap A = \emptyset. $$
	Here we used that $J \cap A \subset J \cap \mathcal{B}$ for $J \subset (0,1)$. This clearly shows that $A$ cannot be super-dense.
\end{proof}

While Proposition~\ref{prop:top1} gives a rather satisfactory necessary condition for a real set to be super-dense, the next result deals with a complementary sufficient criterion 

\begin{proposition}\label{prop:top2}
	Let $A \subset \mathbb{R}$ be a  co-meager   $G_{\delta}$-set. Then, $A$ is super-dense.
\end{proposition}

Proposition~\ref{prop:top2} is due to the anonymous referee.
\begin{proof}
	Let $\varphi : \mathbb{R} \to \mathbb{R}$ be continuous and nowhere constant. We claim that for every open interval $I$ there exists an open subinterval $I_0 \subseteq I$ such that $\varphi(I_0)$ contains an open set $J$ in its closure. Indeed, let us fix an open interval $I$ and we pick an $x \in I$. Since $I$ is open, we find an $\varepsilon > 0$ such that $(x-\varepsilon, x + \varepsilon) \subseteq I$. Since $\varphi$ is nowhere constant, we find a $y \in (x-\varepsilon, x + \varepsilon)$ such that $\varphi(x) \neq \varphi(y)$. We set $I_0 = (x,y)$ and denote by $\overline{I_0} = [x,y]$ its closure.  We employ the intermediate value theorem to deduce that $\varphi(\overline{I_0}) \supseteq [\varphi(x), \varphi(y) ]$ and the latter clearly contains an open interval since $\varphi(x) \neq \varphi(y)$.
	
	We  show that $\varphi(A \cap I_0)$ is co-meager in some open interval. By hypothesis, $A$ is co-meager in $\rr$ and, hence, $A \cap I_0$ is co-meager in $I_0$. Of course,  $A \cap I_0$ is also dense in $I_0$. Since $\varphi$ is continuous with $\varphi(I_0) \supseteq J$, the set $\varphi(A \cap I_0)$ is dense in $J$. Since $A$ is a $G_\delta$ set, $\varphi(A \cap I_0)$ is the continuous image of a Borel set and, consequently, $\varphi(A \cap I_0)$ has the Baire property.
	
	Then there exists an open set $U$ such that the symmetric difference $\varphi(A \cap I_0) \triangle U$ is meager. Since $\varphi(A \cap I_0)$ is dense in $J$, we have that $U$ is dense in $J$, so there is some open interval $J' \subseteq J$ such that $J' \subseteq U$ and $U \cap J'$ is dense in $J'$. This means $\varphi(A \cap I_0) \cap J'$ differs from the dense open set $U \cap J'$ by a meager set, hence $\varphi(A \cap I_0) \cap J'$ is co-meager in $J'$.
	
	We have two sets that are co-meager in the open interval $J'$: $\varphi(A \cap I_0) \cap J'$ (by the argument above) and $A \cap J'$ (by the initial assumption that $A$ is co-meager and, thus, co-meager in every open interval). By the Baire Category Theorem, the intersection of two co-meager subsets of the complete metric space $J'$ is non-empty. Thus,
	\[
	\varphi(A \cap I_0) \cap A \supseteq (\varphi(A \cap I_0) \cap J') \cap (A \cap J') \neq \emptyset.
	\] \vspace{-0.2cm} \end{proof}
In a next step, we combine our necessary and sufficient criteria from Corollary~\ref{cor:top1} and Proposition~\ref{prop:top2} to deduce the full characterization of super-dense with the Baire property, i.e. the first assertion of Theorem~\ref{thm:topology}.
\begin{corollary}\label{cor:baire}
	Suppose $A \subset \rr$ has the Baire property. Then, $A$ is super-dense if and only if $A$ is co-meager.
\end{corollary}

\begin{proof}
	Suppose $A$ has the Baire property. We show both implications separately from each other.
	
	Suppose $A$ is co-meager.
	Since $A^c$ is meager, there exists a countable family of nowhere dense sets $N_k$ such that $A^c = \bigcup_{k = 1}^{\infty} N_k$. By definition, the closures $\overline{N}_k$ are nowhere dense, too. It follows that $\bigcup_{k = 1}^{\infty} \overline{N}_k$ is a meager set and, hence,
	$$ G := \left(\bigcup_{k = 1}^{\infty} \overline{N}_k\right)^{c} = \bigcap_{k = 1}^{\infty} \overline{N}_k^{c}  $$
	is a co-meager $G_\delta$-set which is contained in $A$. By Proposition~\ref{prop:top2}, $G$ and thus $A$ is super-dense.
	
	Now let $A$ be super-dense.
	$A$ has the Baire property, so there exists an open set $U$ and a meager set $N$ such that $ A = U \Delta N.$ We claim that $U$ is a dense open set. Indeed, if $U \cap I = \emptyset $ for an open interval, the Baire property implies that $A \cap I$ is meager. This contradicts the super-density of $A$ in view of Corollary~\ref{cor:top1}. But this implies that $A^c \subset U^c \cup N$, where $U^c$ is closed with empty interior, i.e. $U^c$ is nowhere dense, and by definition $N$ is meager. In total, $A^c$ is meager.
\end{proof}

Our proofs of the remaining assertions of Theorem~\ref{thm:topology} are in some sense more explicit. We recall the popular Vitali construction of a non-measurable real set $V$: define the equivalence relation on $\mathbb{R}$
$$ x \simeq y \iff x-y \in \mathbb{Q}, $$
which naturally decomposes $\rr$ into equivalence classes $[x]$. We select exactly one representative of each equivalence class to form the set $V$. Of course, we made use of the axiom of choice in this construction. It is easy to see that we can assume $V$ to be dense. Indeed let $D \subset V$ be  some infinitely countable set and $(I_{q^{(1)}, q^{(2)}} )_{q^{(1)} < q^{(2)} \in \mathbb{Q}}$ be the countable collection of open intervals with rational endpoints $q^{(1)}, q^{(2)}$. Let $d_1, d_2, \ldots$ an enumeration of $D$ and similarly we fix an enumeration $I_{q_k^{(1)},q_k^{(2)}}$ of the rational open intervals. We replace each $d_k \in D$ by some $d_k + r_k \in I_{q_k^{(1)},q_k^{(2)}}$ with $r_k \in \mathbb{Q}$. 
We fix such dense Vitali set $V$ and define $V_q := \{x + q \, | \, x \in V\}$ for $q \in \mathbb{Q}$. By construction all $V_q$ are pairwise disjoint. We further introduce the following partition of the real line
$$ E := \bigcup_{n \in \zz } [2n, 2n+1) \qquad F: = \bigcup_{n \in \zz } [2n-1, 2n).$$
After these preparations, we are ready to state our next result
\begin{proposition}\label{prop:top3}
	Let $A := \bigcup_{q \in E \cap \mathbb{Q}} V_q$. Then, $A^c = \bigcup_{q \in F \cap \mathbb{Q}} V_q$ and both $A,A^c$ are dense, but not super-dense.
\end{proposition}
\begin{proof}
	The representation for $A^c$ follows from the fact that the sets $V_q$ form a partition of the real line. Density is a direct consequence of $V$ itself being dense. Considering $\varphi(x) := x +1$, we directly see that $\varphi(A) = A^c$ and $\varphi(A^c) = A$. Hence, neither of the sets is super-dense.
\end{proof}

For the final assertion, we need a few set theoretic preliminaries. We denote as usual by $\aleph_0$ the cardinality of $\nn$ and by the $\aleph_1$ the smallest uncountable cardinality, i.e., the cardinality of the set of all countable ordinal numbers. The cardinality of the reals is given by $2^{\aleph_0}$ and under the continuum hypothesis $2^{\aleph_0} = \aleph_1.$ We further recall that any set $X$ can be well-ordered which we denote by $\prec.$ For a fixed well-order $\prec$ and every $x \in X$,  let $I_{\prec}(x) = \{t \in X \, | \, t \prec x\}$.
\begin{lemma}\label{lem:set}
	\begin{enumerate}
		\item The cardinality of the set of nowhere constant continuous function is $$\{f : \rr \to \rr \, | \, f \text{ is continuous and nowhere constant} \} = 2^{\aleph_0}. $$
		\item Let $X$ be a set. Then, there exists a well-order such that  $|I_{\prec}(x)| < |X|$.
	\end{enumerate}
\end{lemma}
\begin{proof}
	\begin{enumerate}
		\item Since the set contains all functions of the form $x + \alpha$ with $\alpha \in \rr$, its cardinality is at least $2^{\aleph_0}$. On the other hand, since every continuous function is uniquely determined by its values on $\mathbb{Q}$ the cardinality is at most $|\rr^{\mathbb{Q}}| = 2^{\aleph_0}$.
		\item This statement is  folklore in set theory (see e.g. \cite{AP17}[Theorem 2]). Suppose a well-order $\prec$ has not the property. Then, there is a (with respect to $\prec$) minimal $y \in X$ such that $|I_{\prec}(y)| = | X |$. By definition of cardinality numbers, there exists a bijection $\tau: I_{\prec}(y) \to X$. The new well-order $ x \prec' x' \Leftrightarrow \tau^{-1}(x) \prec \tau^{-1}(x')$ has the desired property.
	\end{enumerate}
\end{proof}
We are ready to complete the proof of Theorem~\ref{thm:topology}.
\begin{proposition}\label{prop:top4}
	There exists a set  $A \subset \mathbb{R}$ such that both $A, A^c$ are  super-dense.
	
\end{proposition}
\begin{proof}
	Let $F := \{f : \rr \to \rr \, | \, f \text{ is continuous and nowhere constant} \}$ and $\prec$ a well-order on $F$ with the property from Lemma~\ref{lem:set}. Let $q_n$ be an enumeration of the rationals. We construct via transfinite induction disjoint sets $A_f^{1}, A_f^{2}$ with $f \in F$ such that 
	\begin{enumerate}
		\item $|A_f^{i}| < 2^{\aleph_0}$ for every $f \in F$,
		\item $A_{f}^{i} \subset A_{f'}^{i}$ for $f \prec f'$,
		\item for every $g \prec f$ and every $q_n$ there exists an $r \in A_f^{i}$ such that $| r -q_n| < 1/n$ and $g(r) \in A_f^{i}$.
	\end{enumerate}
	Fix $f \in F$ and suppose we have already constructed $A_g^{i}$ for $g \prec f, g \neq f$. We first set 
	$$ B_f^{i} := \bigcup_{g \prec f, g \neq f} A_g^{i},$$
	which by assumption satisfies the first two conditions and the third condition except for the function $f$. The $f$-update is done inductively, i.e., a standard induction on $n$, where at each step only finitely many points are added. Let $A_{f,0}^{i} := B_f^{i}$ and suppose we already know $A_{f,n-1}^{i}$. Since $f$ is nowhere constant, the image $f((q_n - 1/n, q_n + 1/n))$ contains an interval and thus has cardinality $2^{\aleph_0}$. This is still true for $ X = f((q_n - 1/n, q_n + 1/n))\setminus (\cup_{i = 1,2} A^{i}_{f,n-1} \cup f( A^{i}_{f,n-1}))$. We pick $y \in X$ and $x \in (q_n -1/n, q_n +1/n)$, and set $A^{1}_{f,n} := A^{1}_{f,n-1} \cup \{y\}\cup\{x\}$. We repeat the construction for  $A^{2}_{f,n}$ and note that $A^{2}_{f,n}$,$A^{1}_{f,n}$ are still disjoint and have cardinality less than $2^{\aleph_0}$.
	
	Having constructed the sets $A^{i}_{f,n}$, we define $A^{i}_f = \bigcup_{n \in \nn} A^{i}_{f,n}$ and by construction these sets have the desired properties.
	
	We define $A := \bigcup_{f \in F} A^{1}_{f}$ and due to the third property of $A_f^{1}$, the set $A$ is super-dense. Since $A^c$ contains all sets $A_f^{2}$, it is super-dense by the same reasoning.
\end{proof}

\subsection{Proof of Theorem~\ref{thm:main}}\label{sec:proof1}

Let us fix some notation. Given an integer base $b \geq 2$, we write $x = 0.x_1 x_2 x_3 \ldots$ for the $b$-ary expansion of $x \in [0,1)$. For $d \in \{0,1,\ldots,b-1\}$ and $M \in \nn$, we set
\begin{equation}\label{eq:Lambda}
\Lambda_{b,d,M}(x) := \frac{1}{M} \sum_{j=1}^{M} \chi_{d}(x_j)
\end{equation}
with the characteristic function $\chi_d$, that is, $\chi_d(d) = 1$ and $\chi_d(k) = 0$ for $k \neq d$. In other words, $\Lambda_{b,d,M}(x)$ is the density of the digit $d$ in the first $M$ digits of the $b$-ary expansion of $x$. For general $x \in \rr$, we simply set $\Lambda_{b,d,M}(x) := \Lambda_{b,d,M}(|x| \, \mathrm{mod} \,  1).$ We are interested in the set $\mathcal{Z}_{b}$
of numbers with accumulations of zeros, that is,
\begin{equation}\label{eq:zero}
    \mathcal{Z}_{b} := \{x \in \rr \, | \, \limsup_{M \to \infty } \Lambda_{b,0,M}(x) = 1 \}
\end{equation}

Proposition~\ref{prop:main} below gives a constructive demonstration for the super-density of $\mathcal{Z}_b$  and, thus, establishes the first part of Theorem~\ref{thm:main}. 
\begin{proposition}\label{prop:main}
    Let $b \geq 2$ be an integer base, $I \subset \rr$ a non-empty open interval, and  $\varphi \colon I \to \rr$  a nowhere constant continuous function. Then, there is an iterative algorithm computing a real number $x$ with
    \begin{equation}\label{eq:prop}
    	x \in \mathcal{Z}_b \, \text{ and } \, \varphi(x) \in  \mathcal{Z}_b.
    \end{equation}
\end{proposition}

The proof of Proposition~\ref{prop:main} relies on an iterative ``zig-zag" construction of an $x \in \mathcal{Z}_{b}$ such that $y = \varphi(x) \in \mathcal{Z}_{b}$, too. Through a two-step update we ensure that the $b$-ary expansion of $x$ and $y$ are dominated by zeros infinitely often.
\begin{proof} We fix the basis $b \geq 2$ and assume without loss of generality that $I \subset (0,1)$. For an $x \in (0,1)$, we denote by $x = 0.d_1 d_2 \ldots$ its $b$-ary expansion in digits $d_i$.  Let us introduce the set $T$,
\begin{equation}
    T := \{ x \in (0,1) \,|\, \exists L \in \nn \text{ such that } x = 0.d_1 d_2 \ldots d_L  \},
\end{equation}
of numbers with finite $b$-ary expansion and we set 
\begin{equation}
    \tau(x) := \min \{ L \in \nn \, | \, x = 0.x_1\ldots x_L \}
\end{equation}
for $x \in T$ and otherwise $\tau(x) = \infty$. Our construction requires a second set 
\begin{equation}
    T' := \bigcup_{x \in T} \bigcup_{ K > \tau(x)^2 +1}  \left\{ x + \frac{1}{b^{K}} \right\},
\end{equation}
where the inner union runs over all integers $K > \tau(x)^2 +1$.  A number $y \in T'$ in the form $x + b^{-K}$ with $x \in T$ and $K > \tau(x)^2 +1$ and this decomposition is in fact unique as $K$ is determined by the position of the last non-zero digit (which must be an isolated '1'). For a $y = x + \frac{1}{b^{K}} \in T'$, we similarly set $\tau'(y) := K.$ The function $\tau'(y)$ is well-defined by our previous observation. One easily sees that both, $T$ and $T'$, are dense sets in $(0,1)$. The idea behind the definition of $T'$ is that the zero digits of some $y = x + \frac{1}{b^{K}} \in T'$ between the positions $\tau(x)$ and $K$ are preserved under small enough subtractions.

We now construct 4 sequences 
$$
(x_M^{(1)})_{M \in \nn }, \, (x_M^{(2)})_{M \in \nn }, \quad (y_M^{(1)})_{M \in \nn }, \, (y_M^{(2)})_{M \in \nn }
$$
and prove that $x := \lim_{M \to \infty} x_M^{(2)} \in \mathcal{Z}_b, $  $y := \lim_{M \to \infty} y_M^{(2)} \in \mathcal{Z}_b, $ and $y = \varphi(x).$

\textit{Step 1: Choice of starting points and intervals } \\
Since $T$ is a dense set, we find some $x_{1}^{(1)} \in T \cap I$
and define $L_{x,1} := \tau(x_{1}^{(1)}).$ Let  $y_{1}^{(1)} := \varphi(x_{1}^{(1)})$ be the corresponding function value.  Moreover, we introduce the intervals $$I_{1,x} = (x_{1}^{(1)}, x_{1}^{(1)} + b^{-L^{2}_{x,1} - f})  \quad I_{1,y} = (-\infty, \infty), $$ 
where we choose the integer $f \in \nn$ big enough s.t. the closure of $I_{1,x}$ is still contained in $I$.

To construct the two other starting values, we recall that $\varphi$ is nowhere constant, and thus we find some $z \in I_{1,x}$ such that $\varphi(z) \neq \varphi(x_{1}^{(1)}).$ Employing the intermediate value theorem and the density of $T'$, we in fact find some $x_{1}^{(2)} \in I_{1,x} $ such that $y_{1}^{(2)} := \varphi(x_{1}^{(2)}) \in T' $. The digits $x_{1}^{(2)}$ at position $L_{x,1}$ to  $L_{x,1}^2$-th are zero.

\textit{Step 2: Iterative procedure:} \\
Suppose that for some $M \in \nn$, we have already defined the values and intervals $$x_M^{(1)} \in T, \, x_M^{(2)} \in I_{M,x} , \quad y_M^{(1)} = \varphi(x_M^{(1)}), \, y_M^{(2)} = \varphi(x_M^{(2)}) \in T' \cap I_{M,y}.  $$

We describe how to construct the $M+1$-th values and intervals. We set $L_{y,M} := \tau'(y_M^{(2)})$ and choose $K_{y,M} >L_{y,M} $ big enough s.t.

 \[ I_{M+1,y} := \left(y_M^{(2)} - \frac{1}{b^{K_{y,M}}}, \, y_M^{(2)} + \frac{1}{b^{K_{y,M}}}\right) \subset I_{M,y}. \]
The $b$-ary expansion of any $z \in I_{M+1,y} $ agrees with the one of $y_M^{(2)}$ up to the index $L_{y,M}-1$ (here we make use of the stability of $T'$ under subtraction due to the additional '1'). By continuity of $\varphi$ we find some $x_{M+1}^{(1)} \in T \cap I_{M,x}$ with $y_{M+1}^{(1)} := \varphi(x_{M+1}^{(1)}) \in I_{M+1,y}.$ We set $L_{x,M+1} := \tau(x_{M+1}^{(1)})$ 
and $$I_{M+1,x} := (x_{M+1}^{(1)}, x_{M+1}^{(1)} + b^{-L^{2}_{x,M+1} - 1}) \cap I_{M,x} \neq \emptyset.$$
As in the first step, one finds
some $x_{M+1}^{(2)} \in I_{M+1,x} $ such that $y_{M+1}^{(2)} := \varphi(x_{M+1}^{(2)}) \in T' \cap I_{M+1,y}  $. We note that the $L_{x,k}$-th to the $L_{x,k}^2$-th digit of $x_{M+1}^{(2)}$ and the$\lceil\sqrt{L_{y,k}}\rceil$-th to the $L_{y,k}$-th digits of $y_{M+1}^{(2)}$ are zero for any $k = 1, \ldots, M$. \\[0.5ex] 

To finish the proof, we observe that both $x_M^{(2)}$ and $y_M^{(2)}$ are Cauchy sequences and, thus, posses limits $x \in I$ and $y \in \rr$. By continuity of $\varphi$, one has $y = \varphi(x).$ By construction of $x$ one immediately obtains that 
$$ \Lambda_{b,0,L_{M,x}^2}(x) \geq 1 - \frac{1}{L_{M,x}},$$
which implies that $x \in \mathcal{Z}_b$ if $L_{M,x} \to \infty.$  If $L_{M,x}$ remains bounded the limit $x$ has only finitely many non-zero digits and thus $x \in \mathcal{Z}_b,$ too. One similarly shows that $y \in \mathcal{Z}_b,$ which completes the proof.
\end{proof}

 The proof of Theorem~\ref{thm:main} in the case of a countable collection of nowhere constant continuous functions $\varphi_k : I \to \rr$ largely builds on the construction for a single function $\varphi$ presented in the proof of Proposition~\ref{prop:main}. 
 
 \begin{proof}[Proof of Theorem~\ref{thm:main} ]
 Let $b \geq 2$ be some integer base, with loss of generality we take $I = (0,1)$ and let $(\varphi_k)_{k \in \nn}$ be a family of nowhere constant continuous functions on $I$. The construction of a number $x$ such that all $\varphi_k(x)$ are non-normal is essentially the same procedure as in Proposition~\ref{prop:main}, but controlling all $\varphi_k$ makes its description more tedious and notation-heavy.  We fix some map $\sigma : \nn \to \nn$ with the property that the cardinality of its preimages $| \sigma^{-1}(k) | = \infty$ for all $k \in \nn$. A simple and computable example for such a map is given by
 $$ \sigma(N^2 + m) := m +1 \quad  \text{ for } N \in \nn \text{ and } 0 \leq m \leq 2N. $$
 
 At step $m$ of our construction, the function $\varphi_{\sigma(m)}$ is of main interest. We denote by 
 $$ S_m := \sigma^{-1}(\{1, \ldots, m\}) $$
 the so far traversed indices and $L_m(k) := |\{1, \ldots, m\} \cap \sigma^{-1}(k) |$ the number of times the integer $k$ has been visited. The numbers $y^{(k)}_j \in T' $ stand for the $j$-th approximation of $\varphi_k(x)$. We denote by $x^{(m)}$ the approximations of the desired number $x \in I$. 
 The idea is to ensure that $ y_k := \varphi_k(x) = \lim_{j \to \infty} y^{(j)}_k$ and to force $y_{k}$ to be zero-heavy, i.e.,  $y_k \in \mathcal{Z}_b$.
 
 \textit{Step 1: Choice of starting points} \\
 Using continuity and the nowhere constancy of $\varphi_{\sigma(1)}$, we find as  in the proof of Proposition~\ref{prop:main} a real $x^{(1)} \in (1/4,3/4)$ and $y_{\sigma(1)}^{(1)} \in T'$ such that $y_{\sigma(1)}^{(1)} = \varphi_{\sigma(1)}(x^{(1)})$.
 
 \textit{Step 2: Iterative procedure} \\
 Suppose we have already constructed $x^{(m)} \in I$ and approximants  $y_{k}^{(L_m(k))} \in T'$ such that \begin{equation}\label{eq:cond} |y_{k}^{(L_m(k))} - \varphi_k(x_m)| <  b^{-\tau'\left(y_{k}^{(L_m(k)))}\right) - 1}
 \end{equation}
for all $k \in S_m$. Here we used the function $\tau'$ from the proof of Proposition~\ref{prop:main}. Our goal is to find $x^{(m+1)} \in I$ and $y^{(L_{m+1}(\sigma(m+1)))}_{\sigma(m+1)} \in T'$ such that $|x^{(m+1)} - x^{(m)}| < 2^{-(m+2)}$, $\tau'\left(y^{(L_{m+1}(\sigma(m+1)))}_{\sigma(m+1)}\right) > m$ and \eqref{eq:cond} holds true for $m+1$, too.

Denote by $U_m$ be the open set, for which $z \in U_m$ satisfies the condition \eqref{eq:cond} and $| z - x^{(m)}| < 2^{-(m+2)}$. The standard argument involving the nowhere constancy and continuity of $\varphi_{\sigma(m+1)}$ and the density of $T'$, we find an $x^{(m+1)} \in U_m$ such that $y^{(L_{m+1}(\sigma(m+1)))}_{\sigma(m+1)} := \varphi_{\sigma(m+1)}(x^{(m+1)}) \in T'$ with $\tau'\left(y^{(L_{m+1}(\sigma(m+1)))}_{\sigma(m+1)}\right) > m.$

\textit{Step 3: Conclusion} \\
The numbers $x^{(M)}$ form a Cauchy sequence with limit $x$. Since $x^{(1)} \in (1/4,3/4)$ and $|x^{(m+1)} - x^{(m)}| < 2^{-(m+2)}$, we have $x \in (0,1)$. Since for any $k \in \nn$ the sequence $y_k^{(m)}$ forms a subsequence of the function values $\varphi_k(x_m)$, we have by continuity $y_k := \varphi_k(x) = \lim_{m \to \infty} y_k^{(m)}$. On the other hand, $$|y_{k}^{(m+1)} - y_{k}^{(m)}| <  b^{-\tau'\left(y_{k}^{(m)} \right) - 1},  $$
which guarantees that there are infinitely many integers $K \in \nn$ such that all digits of $y_k$ between positions $K$ and $K^2$ are zero. Hence, $y_k \in \mathcal{Z}_b$ for all $k \in \nn$.
 \end{proof}
 
 We continue our discussion of Theorem~\ref{thm:main} by giving a rigorous account on when the above construction of $x$ can be turned to an algorithm.
 \begin{corollary}\label{cor:algo}
 	Suppose the functions $\varphi_k : (0,1) \to \rr$ satisfy for every $k \in \nn$ the following conditions. 
 	\begin{enumerate}
 		\item Each $\varphi_k$ is a computable function.
 		\item  $\varphi_k$ is a $C^1$-function, i.e., continuously differentiable.
 		\item The set $ \{ x \in \rr \, | \, \varphi_k'(x) = 0\} $ is finite and its number of elements and their values are computable.
 		
 	\end{enumerate}
 	Then, there exists an algorithm computing an $x \in (0,1)$ such that all $\varphi_k(x)$ are  non-normal.
 \end{corollary}
 
 The first assumption is essentially minimal and the latter two can certainly be further weakened, but are in all practical satisfied and simplify the discussion of the algorithm.
 \begin{proof}
 	We apply the same construction as in the proof of Theorem~\ref{thm:main}. Hence, it suffices to show that there are computable choices for $y_k^{(m)} \in T'$ and $x^{(m)}$. We only demonstrate that the iterative update is computable. Note that due to the second assumption each $\varphi_k$ is piecewise strictly monotone. Moreover, since all functions $\varphi_k$ are computable, a bisection procedure allows us to compute an $\delta_m > 0$ such that for all $z \in [x^{(m)} - \delta_m, x^{(m)} + \delta_m]$,
 	$$ |y_{k}^{(L_m(k))} - \varphi_k(z)| <  b^{-\tau'\left(y_{k}^{(L_m(k)))}\right) - 1}.    $$
 	Computing in a next step, the solutions to $\varphi'_{\sigma(m+1)}(z) = 0$ to a sufficient precision, we can select $z_1, z_2 \in [x^{(m)} - \delta_m, x^{(m)} + \delta_m]$ such that $\varphi_{\sigma(m+1)}(z_1) \neq \varphi_{\sigma(m+1)}(z_2)$. In a next step, we compute $\varphi_{\sigma(m+1)}(z_i)$ to precision $\varepsilon > 0$ (denoted by $\varphi_{\sigma(m+1), \varepsilon}(z_i)$) such that $$ |\varphi_{\sigma(m+1), \varepsilon}(z_1) - \varphi_{\sigma(m+1), \varepsilon}(z_2) | > 4 \varepsilon. $$
 	Choosing $y^{(L_{m+1}(\sigma(m+1))}_{\sigma(m+1)} \in T'$   such that 
 	\[ \left|y^{(L_{m+1}(\sigma(m+1))}_{\sigma(m+1)} - (\varphi_{\sigma(m+1), \varepsilon}(z_1) + \varphi_{\sigma(m+1), \varepsilon}(z_2))/2  \right| < \varepsilon \]
 	we can be sure that there exists some $w \in  [x^{(m)} - \delta_m, x^{(m)} + \delta_m]$ such that $y^{(L_{m+1}(\sigma(m+1))}_{\sigma(m+1)} = \varphi_{\sigma(m+1), \varepsilon}(w)$. In contrast to the proof of Theorem~\ref{thm:main}, we cannot set $x^{(m+1)} = w$, but it is enough to find  $z$ such that \eqref{eq:cond} holds for $m +1$. We can find such an $x^{(m+1)}$ by a bisection procedure.
 	\end{proof}

\subsection{Proof of Theorem~\ref{thm:normal}}\label{sec:normal}

In this section, we construct an explicit function $\widehat{C}$ that contradicts the super-density of $\mathcal{N}$.  
Let us define 
$$ \mathcal{T}_1 := \{x \in \rr \, | \, \text{ the ternary expansion of } x  \text{ has infinitely many ones} \}$$
and recall that the set of $2$-simply normal numbers is given by
$$ \mathcal{N}_{s,2} = \left\{ x \in \rr \, | \, \lim_{M \to \infty} \Lambda_{2,0,M} (x) = \frac12  \right\}.$$
We set
\begin{equation}
	\mathcal{A} := \mathcal{T}_1 \cap \mathcal{N}_{s,2}.
\end{equation}
Of course, $\mathcal{N} \subset \mathcal{A}$. Hence, $\mathcal{A}^c$ is a null set. We want to show
\begin{proposition}\label{prop:normal}
	There is an explicit function $\widehat{C} : \rr \to \rr $ such that
 \begin{equation}	\widehat{C}(\mathcal{A}) \cap \mathcal{A} = \emptyset
 	\end{equation}
\end{proposition}

Theorem~\ref{thm:normal} is a direct consequence of Proposition~\ref{prop:normal}. We start with a technical preparation. 
\begin{lemma}\label{lem:extra}
	Let $(I_k)_{k \in \nn}$ be a countable collection of closed intervals $I_k \subset [0,1]$ whose interiors are pairwise joint. Let $(g_k)_{k \in \nn}$ a countable collection of continuous functions $g_k : [0,1] \to \rr$ with support $\supp(g_k) \subset I_k$. Suppose further that  $ \|g_k \|_{\infty} = \sup_{x \in [0,1]} |g_k(x)| \to 0 $ as $k \to \infty$. Then, the function $g: [0,1] \to \rr$,
	\begin{equation}
		g(x) := \begin{cases} g_k(x) & \text{ if } x \in I_k \\
			0 & \text{ else. }
		\end{cases}
	\end{equation}
	is well-defined and continuous.
	\end{lemma}
	
\begin{proof}
	We note that for each $x \in [0,1]$ at most one function value $g_k(x)$ is non-vanishing. This follows from the fact that $g_k(x)$ can only be non-zero if $x$ is contained in the interior of $I_k$. Hence, $g$ is well-defined and can be represented as series $g(x) := \sum_{k = 1}^{\infty} g_k(x)$. We show that this series converges uniformly, which directly implies the continuity of $g$. Using again the observation that at most one function $g_k$ contributes at each $x$, we obtain 
	$$ \left\|\sum_{k =M}^{N} g_k \right\|_{\infty} \leq \max_{k = M, \ldots, N} \| g_k \|_{\infty}. $$
	Since $ \|g_k \|_{\infty} \to 0$, the Cauchy criterion implies the uniform convergence.
\end{proof}

    We spell out the proof of Proposition~\ref{prop:normal}.
\begin{proof}[Proof of Proposition~\ref{prop:normal}]
	The proof heavily relies on properties of Cantor's function $C:[0,1) \to [0,1),$ also commonly known as devil's staircase (see Figure~\ref{fig:cantor_enriched}(a)). We recall that $C$ is increasing and continuous \cite[Proposition 2.1]{DMRV06}. For our purposes, the following description of $C$ is handy. Let $x = 0. x_1 x_2 x_3 \ldots$ be the ternary expansion of $x \in [0,1)$. If the ternary expansion of $x$ contains an '1',  there exists a minimal index $m \in \nn$ s.t. $x_m = 1$. Otherwise we set $m = \infty$.  The Cantor function maps $x$ to the number $y = C(x)$  whose binary(!) expansion $y = 0. y_1 y_2 y_3 \cdots$ is given by
	$$ y_i := \begin{cases}
		0, & \text{ if }  i > m, \\
		\min\{x_i,1\}, & \text{ if } i \leq m. 
	\end{cases}$$
	Clearly, $C$ maps all numbers in $\mathcal{T}_1$ to values with finite binary expansion. Here and in the following, a number with a \textit{finite expansion} in base $b$ is a number whose infinite $b$-ary expansion has only finitely many non-zero digits.  It clearly follows \begin{equation}\label{eq:dense}
		C(\mathcal{A} \cap (0,1)) \cap \mathcal{A} = \emptyset.
	\end{equation}
	Unfortunately, $C$ is locally constant at the deleted thirds in the iterative construction of the Cantor set. The main idea is to enrich $C$ by self-similar copies of itself in each interval such that it becomes a nowhere constant function, which is continuous and still satisfies \eqref{eq:dense}.
	
	To this end, it is convenient to work with the symmetrized version $C_s$
	\begin{equation}\label{eq:Cs}
		C_s(x) := \min\{C(x), 1 - C(x) \} = \begin{cases} C(x), & \text{ if } x \leq \frac12, \\
			C(1-x), & \text{ if } x > \frac12, 
		\end{cases}  
	\end{equation}
	which is by construction continuous. The second identity in \eqref{eq:Cs} follows from the point symmetry of Cantor's function, i.e., $C(1-x) = 1 - C(x)$ for all $x \in [0,1]$. The point symmetry also implies that $C_s$   preserves the crucial property of the Cantor function that it maps all numbers with at least a '1' in the ternary expansion to a number with finite binary expansion. Figure~\ref{fig:cantor_enriched}(b) illustrates $C_s$.
	
		\begin{figure}[htbp]
		\centering
		\includegraphics[width= 0.8\textwidth]{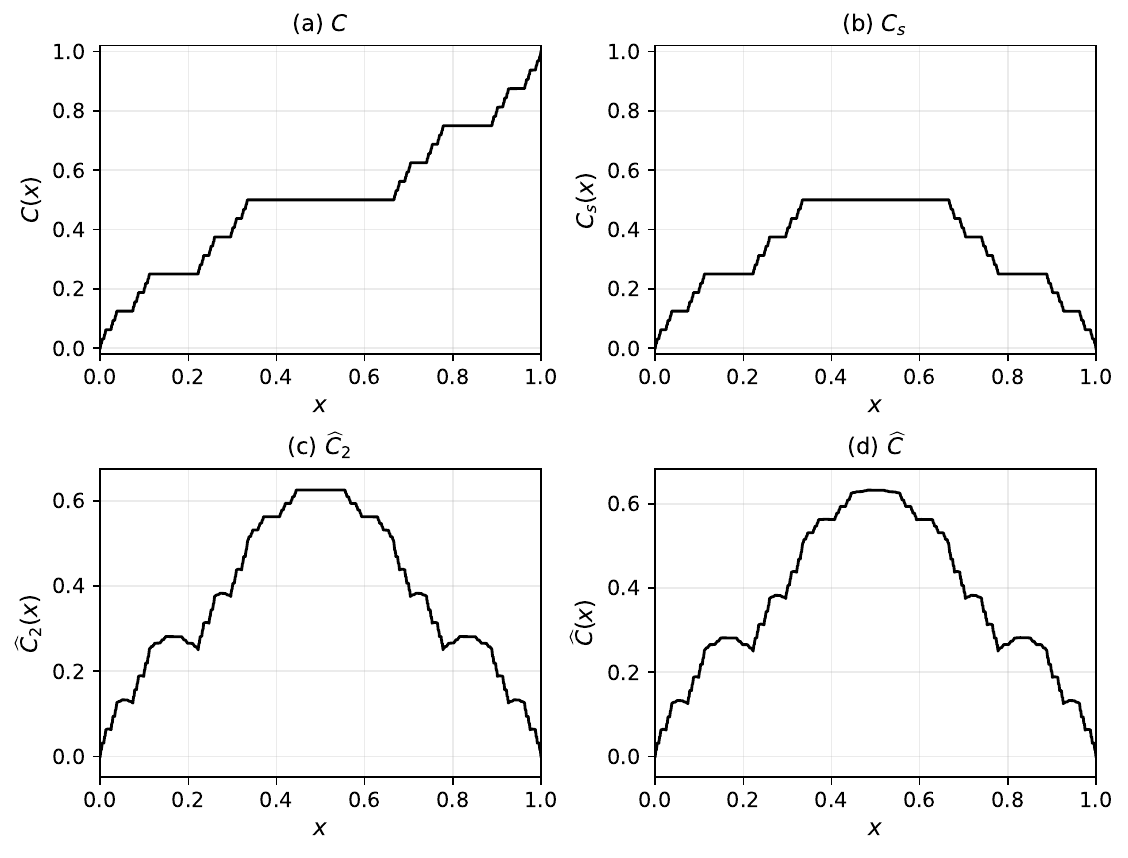}
		\caption{ \small
			Comparison of the classical Cantor function and its enriched variants.\\
			(a) The standard Cantor function $C$, continuous, non-decreasing, and constant on  deleted middle thirds. 
			(b) The symmetrized Cantor function $C_s$. 
			(c) The second approximation $\widehat{C}_2$, obtained by inserting scaled self-similar copies of $C_s$
			after the first occurrence of the digit $1$ in the ternary expansion. 
			(d) The enriched Cantor function $\widehat{C}$ 
			which is continuous and nowhere constant.
		}
		\label{fig:cantor_enriched}
	\end{figure}
	
	We now define the enriched Cantor function $\widehat{C} : [0,1] \to [0,1]$ (see also Figure~\ref{fig:cantor_enriched}(d)). To this end, we distinguish the cases, where $x$ has finitely many '1's and infinitely many ones. If $x$ has no '1', we set $\widehat{C}(x) = C_s(x).$ If $x$ has $m$ '1's in its ternary expansion, we write $x = 0. w_1 w_2 \ldots w_m \sigma, $ where the finite words $w_i$ contain exactly one '1' as their last digit - we call such words \textit{one-stripped} in this proof -  and $\sigma$ is a word without any '1'. Then, we set 
	\begin{equation}\label{eq:cantfinite}
		\widehat{C}(x) := \left(\sum_{r = 1}^{m} 2^{-r\sum_{k =1}^{r-1}|w_r|} C_s(w_r) \right) + 2^{-(m+1)\sum_{k =1}^{r}|w_r|} C_s(\sigma),
	\end{equation}
	where the empty sum is understood as zero. Here, we abuse notation a bit as we identify a ternary word $w$ with the real number $x = 0.w$, i.e. $C_s(w)$ should be understood as $C_s(x)$ with $x = 0.w$.
	In the case $m = 1$, one applies the regular (symmetrized) Cantor function to the first part $w_1$, which gives a number terminating after $|w_1|$ bits. Then, one inserts $|w_1|$-many zeros and continues with $\sigma$ - reinterpreted as bit strings. For higher $m$, one alternates between inserting an increasing number of zeros and inserting the Cantor bit-string of each $w_i$ afterwards. 
	If $x$ has infinitely many '1's in its ternary expansion, we write $x = 0. w_1 w_2 \ldots $ with the same definition for $w_i$ as above and one similarly sets
	\begin{equation}\label{cant:infinite}
		\widehat{C}(x) := \sum_{r = 1}^{\infty} 2^{-r\sum_{k =1}^{r-1}|w_r|} C_s(w_r).
	\end{equation}
	The function $ \widehat{C}(x) $ is nowhere constant as $C(x) \neq C(y)$ for any $y$ with finite ternary expansion and any $x$ with infinitely many '0's and '2's.

	Next, we verify   that $\widehat{C}(x)$ is continuous for which the visualization in Figure~\ref{fig:cantor_enriched}(c) is helpful. We consider the approximations $\widehat{C}_k(x), $ which roughly speaking stop after the $k$-th '1' (similarly to the original Cantor function which stops after the first '1'). A rigorous construction goes as follows. Let $T_k : [0,1) \to [0,1)$ be the truncation map
	$$ T_k(x)  := \begin{cases}
		x & \text{ if the ternary expansion of } x \text{ has less than } k \text{ '1's} \\
		0.w_1w_2 \cdots w_k & \text{ otherwise,}
	\end{cases}   $$ 
	where the $w_i$ are the unique one-stripped words from above with a single '1' as their last digit. We simply set $\widehat{C}_k(x) := \widehat{C}(T_k(x)). $
	Let us show by induction that $\widehat{C}_k(x)$ is continuous for every $k$. Since $\widehat{C}_1(x) = C_s(x)$, this is clear for $k = 1$. Suppose $C_k$ is continuous and we show that $C_{k+1}$ is continuous, too.  The difference $\widehat{C}_{k+1} - \widehat{C}_k$ is nonzero only on closed intervals of the form
	$$ I_{w_1,w_2, \ldots, w_k} := [0.w_1,w_2, \ldots, w_k, w_1,w_2, \ldots, w_k + 3^{-\sum_{j = 1}^{k} |w_j|}]$$
	with one-stripped words $w_i$ and on each such interval we have
	$$ \widehat{C}_{k+1}(0.w_1 \ldots w_k \tau ) - \widehat{C}_{k}(0.w_1 \ldots w_k \tau ) = 2^{-(k+1)\sum_{j =1}^{k}|w_j|} C_s(\tau) $$
	for any ternary word $\tau$. Note that the interiors of the intervals $I_{w_1,w_2, \ldots, w_k}$ are pairwise disjoint and that $\widehat{C}_{k+1}- \widehat{C}_{k}$ is continuous on each $I_{w_1,w_2, \ldots, w_k}$ and vanishes at the respective boundary points. This gives rise to continuous functions $g_{w_1, w_2, \ldots, w_k}$ with support in $I_{w_1,w_2, \ldots, w_k}$. We order the intervals ascendingly in $|w_1| + \ldots + |w_n|$ and since  $\|g_{w_1, w_2, \ldots, w_k}\|_{\infty} \leq 2^{-(k+1)\sum_{j =1}^{k}|w_j|}$ Lemma~\ref{lem:extra} yields the continuity of $\widehat{C}_{k+1}- \widehat{C}_{k} = \sum_{w_1, \ldots, w_k} g_{w_1, w_2, \ldots, w_k}$. By our induction hypothesis, $C_{k+1}$ is continuous.

	 By construction, 
	\begin{equation}\label{eq:Cbound} \sup_{x \in (0,1)}|\widehat{C}(x) -\widehat{C}_k(x)| \leq 2^{-k^2}
	\end{equation}
	so that $\widehat{C}_k$ converges uniformly to $\widehat{C}$. This implies that $\widehat{C}$ is continuous, too.
	
	It remains to show that $\widehat{C}(\mathcal{T}_1) \cap \mathcal{N}_{2,s} = \emptyset. $ To this end, let $x \in \mathcal{T}_1$ and we write for its ternary expansion $x = 0.w_1 w_2 \cdots$ with one-stripped $w_i$. Recall that each $w_i$ is a finite ternary word. Let $\widehat{w_i}$ be the ``Cantor image" of $w_i$ that is the binary string obtained by flipping each '2' in $w_i$ to a '1'. Then, the binary expansion of $\widehat{C}(x)$ has the form $\widehat{C}(x) = \widehat{w_1} z_1 \widehat{w_2} z_2 \cdots$, where each $z_r$ is a string of zeros. Note that after the last digit of $z_m$ the binary expansion of insertion of $\widehat{C}(x)$ has at most $\sum_{r=1}^{m} |w_r| $ '1's but at least $(m-1) \sum_{r=1}^{m} |w_r| $ zeros. That is there a sequence of integers $L_m$ s.t. 
	$$ \Lambda_{2,0,L_m}(\widehat{C}(x)) \geq 1 - \frac1m.$$
	This shows that $\widehat{C}$ has all desired properties on $[0,1]$. Since $\widehat{C}(0) = \widehat{C}(1) = 0$, the periodic extension of $\widehat{C}$ to the whole real line is the desired function.
\end{proof}

\appendix

\section{Proof of $\dim_{\mathcal{H}}(\mathcal{Z}_b) = 0$}

The main purpose of this appendix is to present a proof of the following

\begin{lemma}\label{lem:dim}
	Let $b \geq 2$ be an integer. Then,
	\begin{equation}
		\dim_{\mathcal{H}}(\mathcal{Z}_b) = 0.
	\end{equation}
\end{lemma}
The proof of Lemma~\ref{lem:dim} relies on a standard combinatorial estimate and basic properties of Hausdorff measures. As it does not directly follow from Egglestone's result on the Hausdorff-dimension of sets with specific asymptotic digit frequencies \cite{Egg49},  we choose to sketch the complete argument.

\begin{proof}[Proof of Lemma~\ref{lem:dim}]
	First, we recall the construction of the (outer) Hausdorff measure on the real line. For real numbers $s,\delta > 0$ we introduce the functional
	\begin{equation}
		\mathcal{H}^{s}_{\delta}(A) := \inf \left\{ \sum_{k}^{\infty} |I_k|^s \, \bigg| \, A \subset \bigcup_{k=1}^{\infty} I_k \text{ and } (I_k)_{k} \text{ are open intervals with } |I_k| < \delta \right\}
	\end{equation}
	for any set $A \subset \rr$. Then, the $s$-dimensional (outer) Hausdorff is defined as
	\begin{equation*}
		\mathcal{H}^{s}(A) := \sup_{\delta > 0} \mathcal{H}^{s}_{\delta}(A) = \lim_{\delta \to 0} \mathcal{H}^{s}_{\delta}(A).
	\end{equation*}
	A set $A \subset \rr$ is of zero Hausdorff dimension if and only if $\mathcal{H}^{s}(A) = 0$ for all $s > 0.$ Due to  subadditivity of the Hausdorff measure, it suffices to show
	\begin{equation}\label{eq:epss}
		\mathcal{H}^{s}_{\delta}( \mathcal{Z}_b\cap [0,1))
	\end{equation}
	for all $s, \delta > 0$. From now on, we fix the integer $b \geq 2$ and $s,\delta$. We proceed in two steps.
	
	As a first step, we derive combinatorial bounds on zero-heavy finite words.
	Let us fix some $0 <\varepsilon < \frac12$ and for $M \in \nn$ we consider the set $\Omega_M := \{0,1,\ldots,b-1 \}^M$ of  $b$-ary words of length $M$. For any $w \in \Omega_M$ we set
	$$ N_M(w) = \frac{1}{M} \sum_{j=1}^{M} \chi_{0}(w_j),$$
	and define the set
	$$ \Omega_M^{\varepsilon} := \{ w \in \Omega_M \, | \,N_m(w) \geq 1 - \varepsilon \}. $$
	The size of $\Omega_M^{\varepsilon}$ can be written in terms of a binomial expansion
	\begin{equation} |\Omega_M^{\varepsilon}| = \sum_{K \geq (1-\varepsilon) M}^{M} \binom{M}{K} (b-1)^{M-K} \leq \varepsilon M \binom{M}{\lceil (1-\varepsilon) M \rceil} (1-b)^{\varepsilon M},
	\end{equation}
	where the bound follows from the monotonicity of the binomial. We  employ Stirling's formula to obtain the asymptotic expansion
	$$ \binom{M}{\lceil (1-\varepsilon) M \rceil} = \exp((\gamma(\varepsilon) +o(1)) M),$$
	where we used Landau's $o$-notation and introduced the binary entropy $\gamma : [0,1] \to \rr,$
	\begin{equation}
		\gamma(x) = -x \ln(x) - (1-x) \ln(1-x).
	\end{equation}
	Note that by the above we may find some $M_0 = M_0(\varepsilon)$ s.t. that for all $M \geq M_0(\varepsilon)$
	\begin{equation}\label{eq:boundOmega}
		|\Omega_M^{\varepsilon}| \leq \exp((2 \gamma(\varepsilon) + \varepsilon\ln (1-b))M)
	\end{equation}
	and that the prefactor in the exponential tends to zero as $\varepsilon \to 0$. Therefore, for any $s > 0$, there exists some $\varepsilon(s) $ s.t. for all $\varepsilon < \varepsilon(s)$
	\begin{equation}\label{eq:eps} 2 \gamma(\varepsilon(s)) + \varepsilon(s)\ln (1-b) < \frac{s \ln(b)}{2}  \end{equation}
	
	Now, we construct the desired cover for $\mathcal{Z}_b$. 
	We denote by $\phi_M :[0,1) \to \Omega_M$ the function, which maps $x$ to the word $w$ consisting of the $M$ first digits of $x$ with respect to $b$. The partial inverse $\psi_M :\Omega_M \to [0,1)$ maps a word $w$ to the number $x = 0.w_1 w_2 \ldots w_M.$ By definition,
	\begin{equation}\label{eq:cover}
		\mathcal{Z}_b \cap [0,1) \subset \bigcup_{M = K}^{\infty} \{ x \,| \, N_M(x) \geq 1 - \varepsilon \} = \bigcup_{M = K}^{\infty} \{ x \,| \, \phi_M(x) \in \Omega_{M}^{\varepsilon} \}
	\end{equation}
	for any integer $K$ and any $\varepsilon > 0.$
	Moreover, 
	$$\{ x \,| \, \phi_M(x) \in \Omega_{M}^{\varepsilon} \} \subset \bigcup_{w \in \Omega_{M}^{\varepsilon}} I_M(w) := \bigcup_{w \in \Omega_{M}^{\varepsilon}}\left(\psi_M(w) - \frac{1}{b^M}, \psi_M(w) + \frac{1}{b^M} \right). $$
	
	We choose some $\varepsilon < \varepsilon(s)$ and we consider integers $K$ large enough s.t. $1/b^{K} < \delta$ and \eqref{eq:boundOmega} holds for all $M \geq K.$ Then, 
	\begin{align*} \mathcal{H}^{s}_{\delta}(\mathcal{Z}_b \cap [0,1)) &\leq \sum_{M =K}^{\infty} \mathcal{H}^{s}_{\delta}(\{ x \,| \, \phi_M(x) \in \Omega_{M}^{\varepsilon} \}) \leq \sum_{M =K}^{\infty} \sum_{w \in \Omega_{M}^{\varepsilon}} |I_M(w)|^{s} \\
		& \leq \sum_{M =K}^{\infty} |\Omega_{M}^{\varepsilon}| \frac{2^s}{b^{sM}} \leq 2^s \sum_{M =K}^{\infty} b^{-\frac{s}{2}M} = \frac{2^s}{(1-b^{s/2})b^{Ks/2}}.
	\end{align*}
	The first inequality is a consequence of \eqref{eq:cover} and the fact that the intervals $(I_M(w))_{w \in \Omega_{M}^{\varepsilon}} $ cover the set $\{ x \,| \, \phi_M(x) \in \Omega_{M}^{\varepsilon} \}$. The second line then follows from the choice $\varepsilon < \varepsilon(s),$ and equations \eqref{eq:boundOmega} and \eqref{eq:epss}. 
	As we can choose $K$ as large as we want, the last bound reveals $\mathcal{H}^{s}_{\delta}(\mathcal{Z}_b \cap [0,1)) = 0$, which completes the proof.

\end{proof}

\section{Non-normality for all bases}

 In this appendix, we discuss the simultaneous non-normality for all integer bases.
\begin{corollary}\label{cor:allb}
	Let $I \subset \rr$ be any non-empty open interval and $\varphi \colon I \to \rr$ a nowhere constant continuous  function. Then, 
	\begin{equation}
		\varphi\left(I \cap \bigcap_{b = 2}^{\infty}  \mathcal{N}_{b,s}^{c}\right) \cap \left(\bigcap_{b = 2}^{\infty} \mathcal{N}_{b,s}^{c} \right) \neq \emptyset
	\end{equation}
\end{corollary}

\begin{proof}
	A constructive proof follows essentially the lines of the proof of Theorem~\ref{thm:main}. Instead, we give a topological proof using Theorem~\ref{thm:topology}. Recall the co-meager $G_{\delta}$-set $G_b$ from \eqref{eq:Gb}. 
	We simply define $G := \bigcup_{b \geq 2} G_b$, which is still a co-meager $G_{\delta}$-set. By Theorem~\ref{thm:topology}, $G$ is super-dense. Since $G_b \subset \mathcal{Z}_b$, we also have $G \subset  \bigcap_{b = 2}^{\infty}  \mathcal{N}_{b,s}^{c}$, which proves the super-density of $\bigcap_{b = 2}^{\infty}  \mathcal{N}_{b,s}^{c}$.
\end{proof}

\section*{Acknowledgments}

The author would like to thank Verónica Becher and Christoph Aistleitner for their support and for providing valuable references. Special thanks are due to the anonymous referee for numerous helpful comments on the structure and the content of this work that substantially improved the paper. In particular, the referee inspired Theorem~\ref{thm:topology} and provided  proof ideas for its first statement. More precisely, the proofs of Proposition~\ref{prop:top2} and Corollary~\ref{cor:baire} follow his suggestions.

\end{document}